\documentclass[12pt]{article}
\usepackage[cp1251]{inputenc}

\usepackage{amssymb, amsthm, amsmath, amsfonts}
\usepackage{graphicx,amscd}
\usepackage{amsopn}
\usepackage[T2A]{fontenc}
\usepackage{float}
\usepackage{subfigure}
\usepackage{enumitem}

\newtheorem{theorem}{Theorem}

\newtheorem{lemma}{Lemma}[section]
\newtheorem{proposition}{Proposition}[section]
\newtheorem{statement}{Statement}[section]

\newcommand{\im}{\mathop{\rm im}\nolimits}
\newcommand{\id}{\mathop{\rm id}\nolimits}

\newcommand{\Hom}{\mathop{\rm Hom}\nolimits}
\newcommand{\Ind}{\mathop{\rm Ind}\nolimits}

\begin{document}
{\bf On a structure of non-wandering set of an $\Omega$-stable 3-diffeomorphism possessing a hyperbolic attractor}

{Marina Barinova, Olga Pochinka, Evgeniy Yakovlev, HSE University}

\begin{abstract} This paper belongs to a series of papers devoted to the study of the structure of the non-wandering set of an A-diffeomorphism. We study such set $NW(f)$ for an $\Omega$-stable diffeomorphism $f$, given on a closed connected 3-manifold $M^3$. Namely, we prove that if all basic sets in $NW(f)$ are trivial except attractors, then every non-trivial attractor is either one-dimensional non-orientable or two-dimensional expanding.
\end{abstract}

\section{Introduction and formulation of results}
Let $M^n$ be a smooth closed connected $n$-manifold with a Riemannian metric $d$ and $f:M^n\to M^n$ be a diffeomorphism.  A set $\Lambda\subset M^n$ is called an \textit{invariant set} if $f(\Lambda)=\Lambda$.
An invariant compact set $\Lambda \subset M^n$ is called  \textit{hyperbolic} if there is a continuous $Df$-invariant splitting of the tangent bundle $T_\Lambda M^n$ into {\em stable} and {\em unstable subbundles}
$E^s_\Lambda\oplus E^u_\Lambda$, $\dim E^s_x + \dim E^u_x = n$ ($x\in \Lambda$) such that for $i>0$ and for some fixed $C_s>0$, $C_u>0$, $0<\lambda <1$
    \[\Vert Df^i(v)\Vert \leq C_s\lambda ^i\Vert v\Vert, \quad v\in E^s_{\Lambda},\]
    \[\Vert Df^{-i}(w)\Vert \leq C_u\lambda ^i\Vert w\Vert, \quad w\in E^u_{\Lambda}.\]

The hyperbolic structure of $\Lambda$ implies the existence of stable and unstable manifolds $W^s_x$, $W^u_x$ respectively for any point $x\in \Lambda$:
 \[W^s_x = \{y\in M^n: \lim_{j\to +\infty} d(f^j(x),f^j(y))= 0 \}, \]
 \[W^u_x = \{y\in M^n: \lim_{j\to +\infty} d (f^{-j}(x),f^{-j}(y))= 0\}, \]
which are smooth injective immersions of the $E_x^s$ and $E_x^u$ into $M^n$. Moreover, $W^s_x$, $W^u_x$ are tangent to $E_x^s$ and $E_x^u$ at $x$ respectively. For $r > 0$ we will denote by $W^s_{x,r}$, $W^u_{x,r}$ the immersions of discs on the subbundles $E_x^s$,  $E_x^u$ of the radius $r$.

Recall that a point $x\in M^n$ is {\em non-wandering} if for any neighborhood $U$ of $x$ the inequation $f^n(U)\cap U\ne \emptyset $ holds for infinitely many integers $n$. Then $NW(f)$, the {\em non-wandering set} of $f$, defined as the set of all non-wandering points, is an $f$-invariant closed set.

If the non-wandering set $NW(f)$ of $f$ is hyperbolic and periodic points are dense in $NW(f)$ then $f$ is called {\it an $A$-diffeomorphism} \cite{Smale1967}. In this case the non-wandering set is a finite union of pairwise disjoint sets, called {\it basic sets}
$$NW(f)=\Lambda_1\sqcup\dots\sqcup\Lambda_m,$$
each of which is compact, invariant and topologically transitive. A basic set $\Lambda_i$ of an A-diffeomorphism $f:M^n\to M^n$ is called {\it trivial} if it coincides with a periodic orbit and {\it non-trivial} in the opposite case.

By \cite{Bowen}, every non-trivial basic set $\Lambda_i$, similarly to a periodic orbit, is uniquely expressed as a finite union of compact subsets
$$\Lambda_i=\Lambda_{i_1} \sqcup \dots \sqcup \Lambda_{i_{q_i}},q_i \geqslant 1$$  such that $f^{q_i}(\Lambda_{i_j})
=\Lambda_{i_j}, f(\Lambda_{i_j}) = \Lambda_{i_{j+1}},\,j\in\{1,\dots,q_i\}\,
(\Lambda_{i_{q_i+1}} = \Lambda_{i_1})$. These subsets $ \Lambda_{i_{q_i}},\,q_i \geqslant 1$ are called {\em periodic components} of the set
$\Lambda_i$\footnote{R. Bowen  \cite{Bowen} called these components $C$-dense.}.  For every point
$x$ of a periodic component
$\Lambda_{i_j}$ the set
$W^{s}_x\cap \Lambda_{i_j}$ ($W^{u}_x\cap
\Lambda_{i_j}$) is dense in $\Lambda_{i_j}$.

Without loss of generality, everywhere below we will assume that $\Lambda_i$ consists of a unique periodic component and, in addition, $f|_{W^u_{\Lambda_i}}$ preserves orientation if $\Lambda_i$ is trivial.

A sequence of basic sets $\Lambda_1,\dots,\Lambda_l$ of an $A$-diffeomorphism $f:M^n\to M^n$ is called {\it a cycle} if $W^s_{\Lambda_i}\cap W^u_{\Lambda_{i+1}}\neq\emptyset$ for $i=1,\dots,l$, where $\Lambda_{l+1}=\Lambda_1$. A-diffeomorphisms without cycles form the set of {\it $\Omega$-stable} diffeomorphisms; if, in addition, the stable and the unstable manifolds of every non-wandering point
intersect transversaly then $f$ is {\it structurally stable}  (see, for example, \cite{Robinson1999}).

A non-trivial basic set $\Lambda_i$ is called {\it orientable} if for any point $x\in\Lambda_i$ and any fixed numbers $\alpha > 0,\,\beta> 0$ the intersection index\footnote{Let $J^k:\mathbb{R}^k\rightarrow{M}^3$ be immersions, $D^k$ be open balls of finite radii in $\mathbb{R}^k$, $k=1,2$.
Then the restrictions $J^k:D^k\rightarrow{M}$ are embeddings and their images $W^k=J^k(D^k)$ are smooth embedded submanifolds of the manifold $M^3$. Let $U^k$ be a  tubular neighborhood of $W^k$, which are images of embeddings in $M^3$ of spaces of $(3-k)$-dimensional vector bundles on $W^k$ \cite[Chapter~4, par.~5]{Hirsch}.
Since the balls $D^k$ are contractible, then these bundles are trivial and, hence,  $U^2\setminus{W}^2$ consists of two connected components $U^2_+$ and $U^2_-$. It allows to define a function $\sigma:U^2_+\cup U^2_-\rightarrow\mathbb{Z}$, such that $\sigma(x)=1$ if $x\in U^2_+$ and $\sigma(x)=0$ if $x\in U^2_-$. If  submanifolds $W^1$ and $W^2$ intersect transversally at a point $x=J^1(t)$, $t\in D^1$ then there exists a number $\delta>0$ such that $J^1(t-2\delta,t+2\delta)\subset U^2$. The number
\[\Ind_x(W^1,W^2)=\sigma(t+\delta)-\sigma(t-\delta)\]
 is called an {\it intersection index} of submanifolds $W^1$ and $W^2$ in the point $x$. Notice, that this definition does not require orientability of the manifold $M^3$.} $W^u_{x,\alpha}\cap W^s_{x,\beta}$ is the same at all intersection points
($+1$ or $-1$) \cite{Grines1975}. Otherwise, the basic set is called {\it non-orientable}.

A basic set $\Lambda_i$ is called an {\it attractor} if there exists a compact neighborhood $U_{\Lambda_i}$ ({\it a trapping neighborhood}) of $\Lambda_i$ such that $f(U_{\Lambda_i})\subset {\rm int}\, U_{\Lambda_i}$ and $\Lambda_i=\bigcap\limits^{\infty}_{i=0}f^i(U_{\Lambda_i})$.  Due to \cite{Williams74}, a non-trivial attractor $\Lambda_i$ of $f$ is said to be {\it expanding} if $\dim\,\Lambda_i=\dim\,W^u_x$, $x\in \Lambda_i$.

The main result of this paper is following.
\begin{theorem}\label{mt} Let $f:M^3\to M^3$ be an $\Omega$-stable diffeomorphism whose basic sets  are trivial except attractors. Then every non-trivial attractor is either one-dimensional non-orientable or two-dimensional expanding.
\end{theorem}
\begin{figure}[h!]
\centerline{\includegraphics
[width=9 true cm]{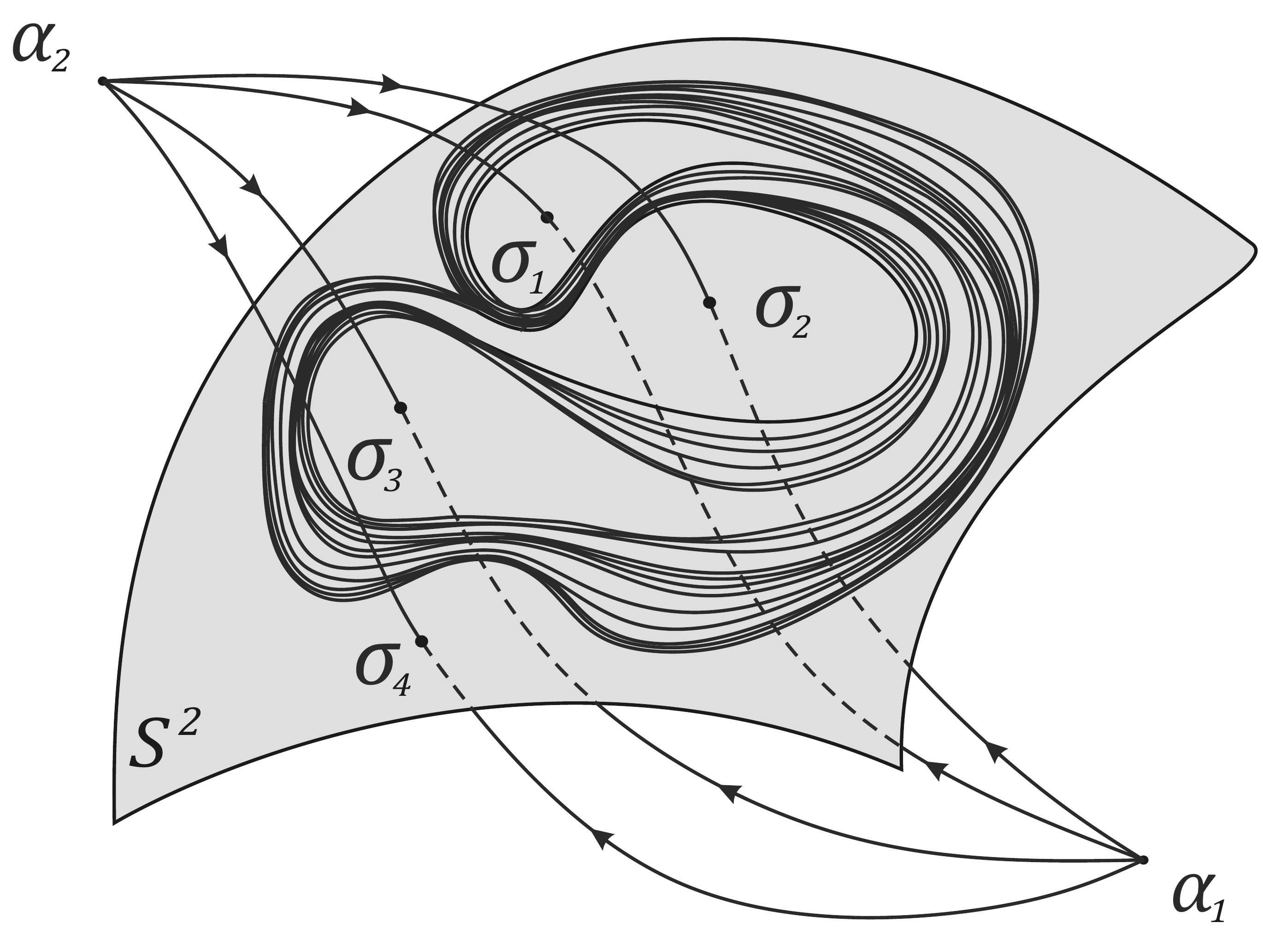}}
\caption{\small $\Omega$-stable diffeomorphism $f:\mathbb S^3\to\mathbb S^3$ with the unique non-trivial basic set which is Plykin attractor}\label{plykin_attractor}
\end{figure}
Notice, that the attractors of both types described in the Theorem \ref{mt} are realized. In particular, the Figure \ref{plykin_attractor} shows a phase portrait of a structurally stable diffeomorphism of a 3-sphere, whose non-wandering set consists of a one-dimensional non-orientable Plykin attractor, four saddle points with a two-dimensional unstable manifold and two sources. The DA-diffeomorphism of 3-torus on Figure \ref{2-dim_expanding_gray} is an example of a combination of an orientable two-dimensional expanding attractor with a source in the non-wandering set of a structurally stable diffeomorphism. An example of a diffeomorphism with non-orientable 2-dimensional expanding attractor will be constructed in section \ref{sec:example}.
\begin{figure}[h!]
\centerline{\includegraphics
[width=10cm]{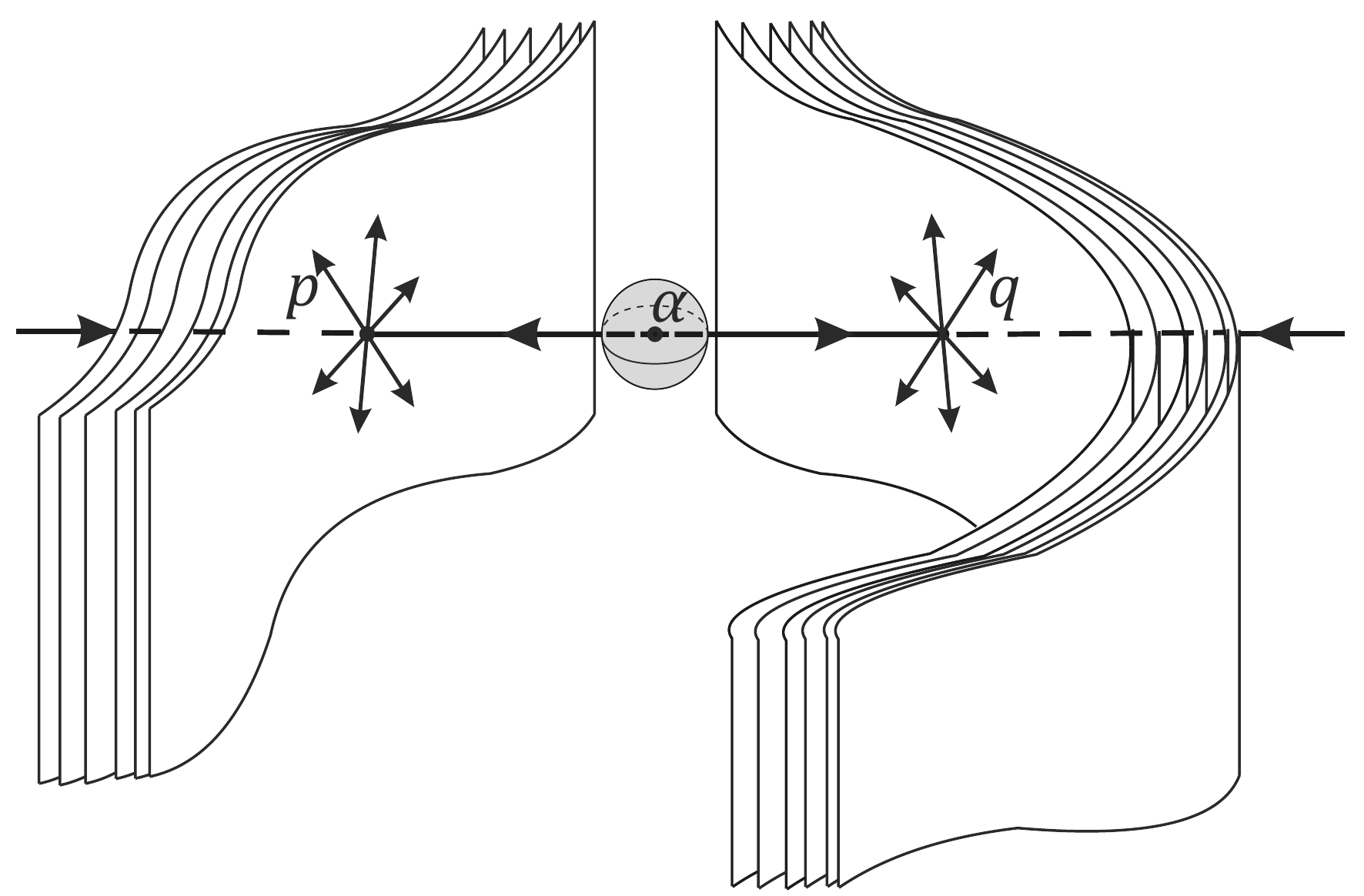}}
\caption{\small DA-map on $\mathbb T^3$}\label{2-dim_expanding_gray}
\end{figure}

\textit{Acknowledgments}. This work was supported by grant 22-11-00027, except section \ref{filtr}, whose results were  supported by the Laboratory of Dynamical Systems and Applications NRU HSE, by the Ministry of Science and Higher Education of the Russian Federation (ag. 075-15-2022-1101).

\section{Attractor, index of a hyperbolic point, filtration}
\subsection{Attractors of an A-diffeomorphism $f:M^3\to M^3$}
Let $f:M^3\to M^3$ be an $A$-diffeomorphism and $\Lambda_i$ be its  basic set. Then $$\dim\,W^u_x+\dim\,W^s_x=3,\,x\in\Lambda_i.$$ If $\Lambda_i$ is a non-trivial then, moreover, $\dim\,W^u_x>0,\,\dim\,W^s_x>0$.

Now let $\Lambda_i$ be a non-trivial attractor. It follows from  \cite{Plykin1974} that $$\Lambda_i=\bigcup\limits_{x\in \Lambda_i}W^u_x$$ and, hence, $\dim\,\Lambda_i>0$.

If  $\bf{\dim\,\Lambda_i=3}$ then $\Lambda_i=M^3\cong\mathbb T^3$ \cite{Newhouse}.

If ${\bf \dim\,\Lambda_i=2}$ then $\Lambda_i$ is either expanding (as in the Figure \ref{2-dim_expanding_gray}) or an {\it Anosov torus} ($f|_{\Lambda_i}$ is conjugate to an Anosov algebraic automorphism of a torus $\mathbb T^2$) \cite{Brown2010}, \cite{GrMeZh2005}.
Herewith, an expanding attractor $\Lambda_i$ is locally homeomorphic to the product of $\mathbb{R}^{2}$ with a cantor set  \cite{Plykin1971,Plykin1984}. There are both type of such attractor,  orientable and non-orientable \cite{MeZhu2002}.
By \cite{GrMeZh2005} every Anosov torus $\Lambda_i$ is a locally flat (possible non-smoothly \cite{KaMaYo1984}) embedded in $M^3$ and, hence, it is always orientable and has a trapping neighborhood $U_{\Lambda_i}$ which is homeomorphic to $\mathbb T^2\times[-1,1]$.

If ${\bf \dim\,\Lambda_i=1}$ then $\Lambda_i$ is automatically expanding, derived from an expansions on a 1-dimensional branched manifold \cite{Williams74} and is the nested intersections of handlebodies \cite{Bothe}. Thus, any one-dimensional attractor $\Lambda_i$ of an A-diffeomorphism $f:M^3\to M^3$ has a trapping neighborhood $U_{\Lambda_i}$ which is a handlebody. There are both type of such attractor,  orientable and non-orientable, it is enough to consider $f=f_{DA}\times f_{NS}$ (see Figure \ref{1-dim_surface_DA}) and $f=f_{Pl}\times f_{NS}$, where $f_{DA}:\mathbb T^2\to\mathbb T^2$ is derived from Anosov diffeomorphism, $f_{NS}:\mathbb S^1\to\mathbb S^1$ is a ``source-sink'' diffeomorphism, $f_{Pl}:\mathbb S^2\to\mathbb S^2$ is a  diffeomorphism with the Plykin attractor (as in the Figure \ref{plykin_attractor}) and four sources.
\begin{figure}[h!]
\centerline{\includegraphics
[width=10 true cm]{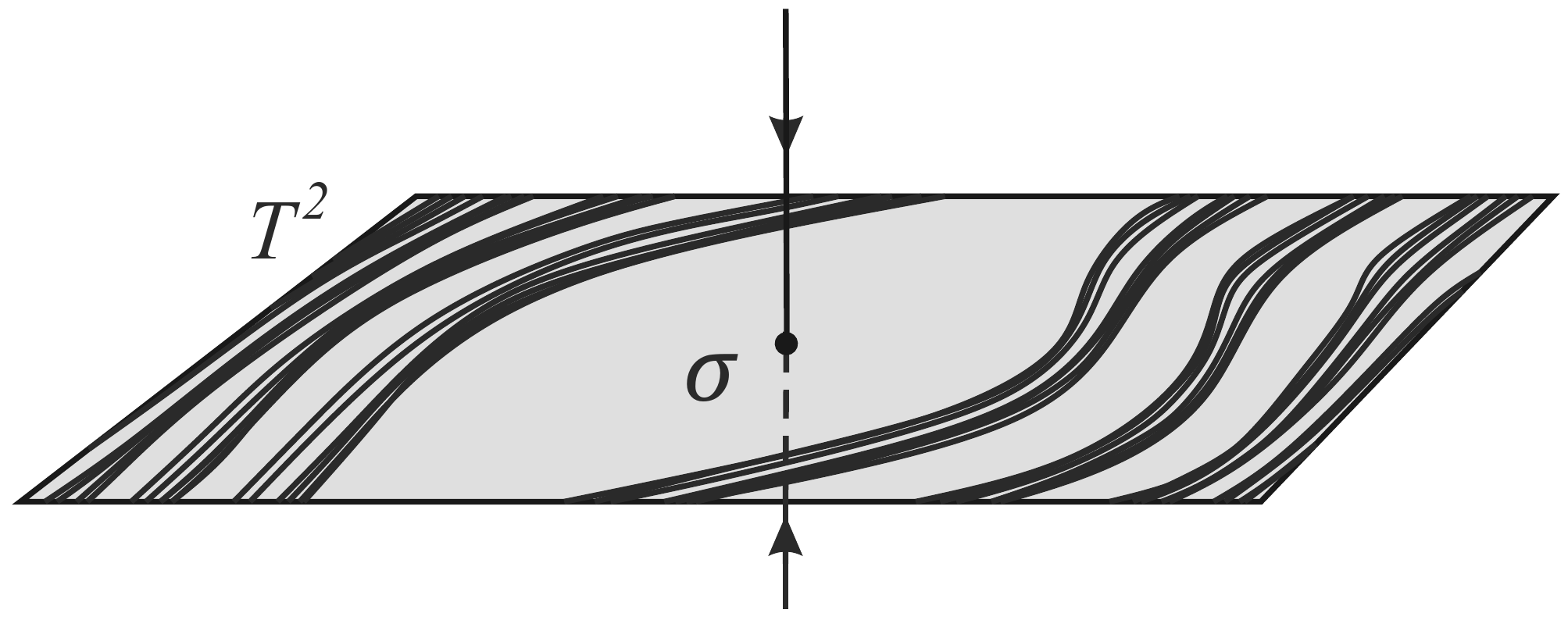}}
\caption{\small A one-dimensional attractor for a diffeomorphism $f_{DA}\times f_{NS}$}\label{1-dim_surface_DA}
\end{figure}

The most famous  one-dimensional attractor is {\it Smale solenoid} (see Figure \ref{solenoid}) which appears as intersection of the nested tori $f^k(\mathbb D^2\times\mathbb S^1),\,k\in\mathbb N$ for $f(d,z)=(d/10,2z)$. An arbitrary one-dimensional attractor is sometimes called {\it Smale-Williams solenoid}.
\begin{figure}[h!]
\centerline{\includegraphics
[width=8cm]{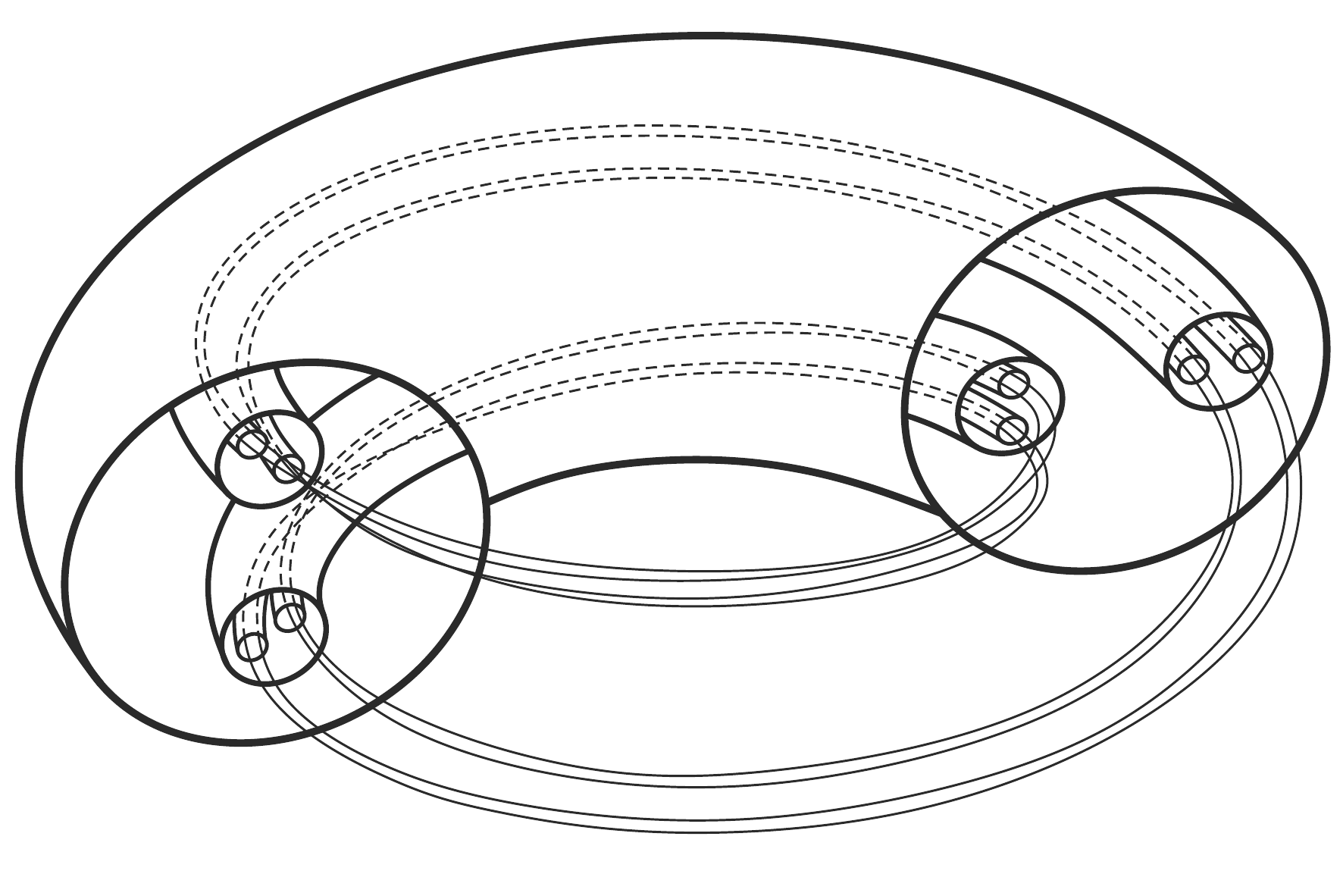}}\caption{\small Smale's solenoid}\label{solenoid}
\end{figure}

It is well known that the presence of an attractor with certain properties in a non-wandering set of an A-diffeomorphism can determine both the cha\-racter of the remaining basic sets and the topology of the ambient manifold.
\begin{itemize}
\item If $f: M^3\to M^3$ is a structurally stable diffeomorphism whose non-wandering set $NW(f)$ contains a two-dimensional expanding  attractor $\Lambda_i$, then it is  orientable, $M^3\cong\mathbb{T}^3$ and the set $NW(f)\setminus \Lambda_i$ consists of a finite number of isolated sources and saddles \cite{GrZh2002}, \cite{MeZhu2002}.
\item If $f: M^3\to M^3$ is an A-diffeomorphism whose every basic set is two-dimensional then its attractors are either all Anosov tori or all expanding \cite{BGPZ}.
\item If $f: M^3\to M^3$ is a structural stable diffeomorphism whose every basic set is two-dimensional then its attractors are all Anosov tori and $M^3$ is a mapping torus \cite{GrLeMePo2015}.
\item  An orientable manifold $M^3$ admits an A-diffeomorphism $f: M^3\to M^3$ with the non-wandering set which is a union of finitely many Smale solenoids if and only if $M^3$ is a Lens space $L_{p,q},\,p\neq 0$. Every  such a diffeomorphism is not structurally stable \cite{Jia}.
\end{itemize}

\subsection{Orientability of the basic set and index of the hyperbolic point}
In this section let $M$ be a compact smooth $n$-manifold $M$ (possibly with a non-empty boundary) and $f:M\to f(M)$ be a smooth embedding of a compact $n$-manifold $M$ to itself and $Fix(f)$ be its set of the fixed points.

Let $p\in Fix(f)$ be an isolated hyperbolic point. By \cite[Proposition 4.11]{Smale1967} the {\it index} $I(p) =I(p, f)$ of $p$ is defined by the formula $$I(p) = (-1)^{\dim\,W^u_p}\Delta_p,$$
where $\Delta_p=+1$ if $f$ preserves orientation on $W^u_p$
and $\Delta_p=-1$ if $f$ reverses it.

\begin{lemma}\label{ind} If $\Lambda_i$ is an orientable hyperbolic attractor with $\dim\, W^u_x=1,x\in\Lambda_i$ for $f$ then $I(p)=I(q)$ for any $p,q\in (Fix(f)\cap\Lambda_i)$.
\end{lemma}
\begin{proof} Suppose the contrary: there are different points $p,q\in (Fix(f)\cap\Lambda_i)$ such that $I(p)=-I(q)$. As $p,\,q$ belongs to the same basic set $\Lambda_i$ then $\dim\,W^u_p=\dim\,W^u_q$ and, hence, $\Delta_p=-\Delta_q$.  Let us assume for the definiteness that $\Delta_q=-1$ and $\Delta_p=+1$. As $\Lambda_i$ is an attractor then $W^u_p,\,W^u_q\subset\Lambda_i$, moreover, ${\rm cl}\,W^u_p={\rm cl}\,W^u_p=\Lambda_i$. Denote by $\ell^1_p,\ell^2_p;\,\ell^1_q,\ell^2_q$ the connected components  of the sets $W^u_p\setminus p;\,W^u_q\setminus q$. By \cite{GrMePo2016} every such a component is dense in $\Lambda_i$. Due to hyperbolicity of $\Lambda_i$ there is a point $x_1$ of the transversal intersection $\ell^1_q\cap W^s_p$ (see Figure \ref{index_for_orientable_attractor}).
\begin{figure}[h!]
\centerline{\includegraphics
[width=10 true cm]{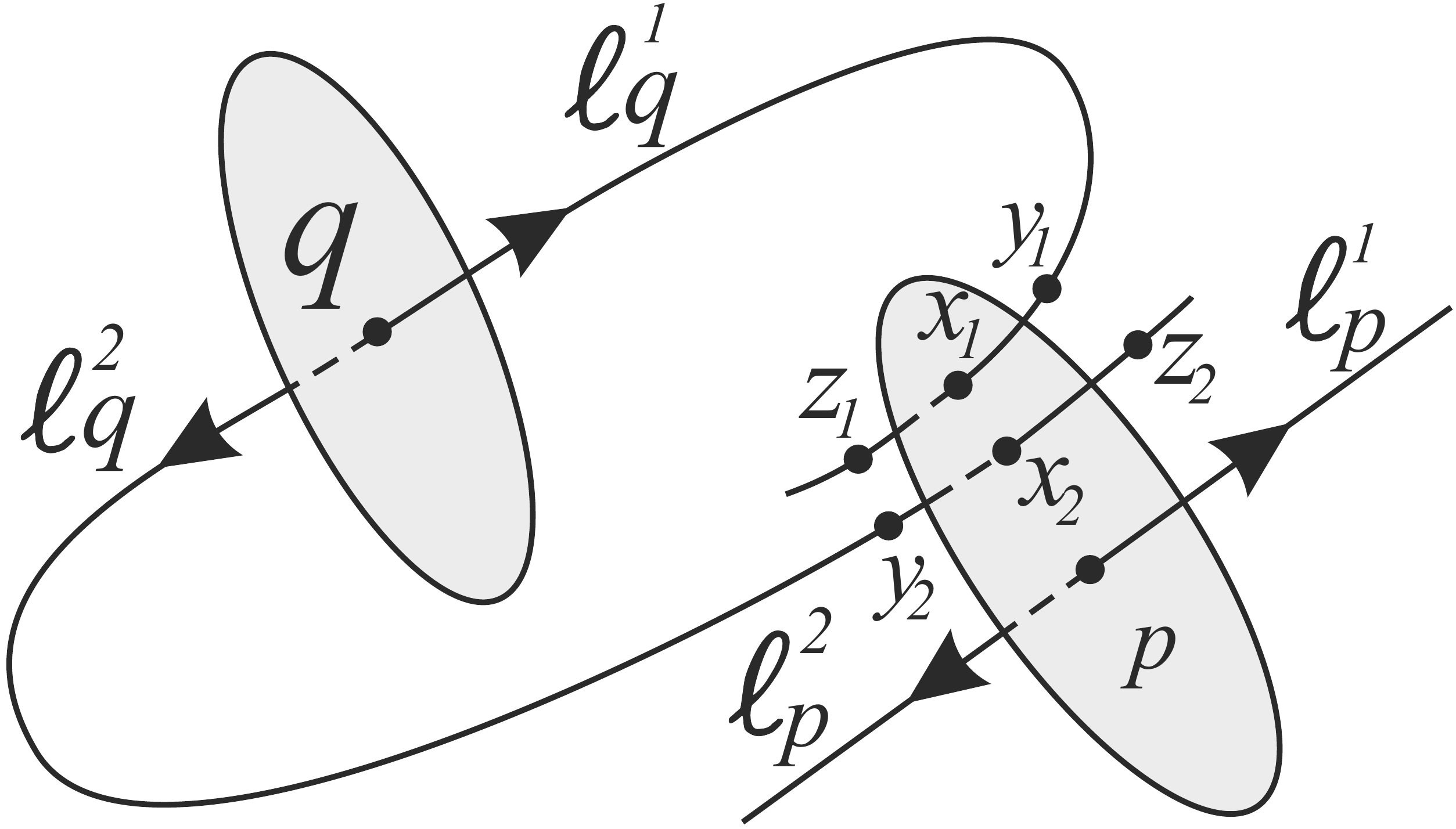}}
\caption{\small Illustration to the proof of Lemma \ref{ind}}\label{index_for_orientable_attractor}
\end{figure}

As $\Delta_q=-1$ then $x_2=f(x_1)$ belongs to $\ell^2_q$. Let $(y_1,z_1)\subset\ell^1_q$ be a neighbourhood of the point $x_1$  and $y_2=f(y_1),\,z_2=f(z_1)$. Then the arc $(y_2,z_2)\subset\ell^2_q$ be a neighbourhood of the point $x_2$. By the orientability of $\Lambda_i$ we get that $y_1,\,y_2$ are separated by $W^s_p$. By $\lambda$-lemma (see, for example,  \cite{Robinson1999}) the iteration of $(y_1,z_1),\,(y_2,z_2)$ with respect to $f$ are $C^1$-closed to $W^u_p$. By continuity of $f$ we conclude that $f(\ell^1_p)=\ell^2_p$. Thus, $\Delta_p=-1$, that contradicts to the assumption.
\end{proof}

Denote by $f_{*k}: H_k(M)\to H_k(M),\,k\in\{0,\dots,n\}$  the induced automorphism of the $k$-th homology group $H_k(M)$ of $M$ with real coefficients. The number
$$\Lambda(f)=\sum\limits_{k=0}^n(-1)^k {\rm tr}(f_{*k})$$ is called a {\it Lefschetz number} of $f$ \cite{Do}.

Suppose $f$ has only
hyperbolic fixed points and their set $Fix(f)$ is finite.
The following equality is named {\it Lefschetz-Hopf theorem}.
\begin{equation}\label{Lef}
\sum\limits_{p\in Fix(f)}I(p)=\Lambda(f).
\end{equation}
Denote by $N_m,\,m\in\mathbb N$ the number of points in $Fix(f^m)$. Let $\lambda_{*k,j},\,j\in\{1,\dots,\dim\,H_k(M)\}$ be  eigenvalues of $f_{*k}$. If $I(p,f^m)=I(q,f^m)$ for any $p,q\in Fix(f^m)$ then the
Lefschetz-Hopf theorem has the following form
\begin{equation}\label{Lef+}
N_m=\left|\sum\limits_{k=0}^n(-1)^k\left(\sum\limits_{j=1}^{\dim\,H_k(M)}\lambda_{*k,j}^m\right)\right|.
\end{equation}

Sometimes it is convenient to pass from homology groups to cohomology groups. Let us prove the following lemma for this aim.

\begin{lemma}\label{lambda_k}
Let $M$ be an $n$-dimensional orientable smooth manifold with boundary $\partial{M}$, $f:M\rightarrow{M}$ be a diffeomorphism, $k\in\{0,1,\dots,n\}$,
$f_*:H_k(M)\rightarrow H_k(M)$, $\tilde{f}_*:H_{n-k}(M,\partial{M})\rightarrow H_{n-k}(M,\partial{M})$
and $f^*:H^k(M)\rightarrow H^k(M)$ be induced automorphisms for groups with real coefficients. Then:
\begin{itemize}
\item if $\lambda$ is an eigenvalue for $f_*$, then $\tilde\lambda=\pm\lambda^{-1}$ is an eigenvalue for $\tilde{f}_*$;
\item if $\tilde\lambda$ is an eigenvalue for $\tilde{f}_*$, then $\lambda=\pm\tilde\lambda^{-1}$ is an eigenvalue for $f^*$.
\end{itemize}
In the both cases a sign $+$ corresponds to an orientation-preserving diffeomorphism and a sign $-$ is used in the opposite situation.
\end{lemma}

\begin{proof}
According to the strong part of the Poincare-Lefschetz duality groups $H_k(M)$ and  $H_{n-k}(M,\partial{M})$ have bases $e_1,\dots,e_m$ and $\varepsilon_1,\dots,\varepsilon_m$, dual with respect to the intersection form
$\Ind:H_k(M)\times H_{n-k}(M,\partial{M})\rightarrow\mathbb{R}$. The duality means that the following equalities take place
$$
\Ind(e_i,\varepsilon_j)=\delta_{ij}, \quad i,j=1,\dots,m.
$$

Let $A$ and $B$ be matrices of automorphisms $f_*$ and $\tilde{f}_*$ in the bases $e_1,\dots,e_m$ and $\varepsilon_1,\dots,\varepsilon_m$ correspondingly. Then
$$
f_*(e_i)=\sum_{s=1}^m a_{is}e_s,\quad \tilde{f}_*(\varepsilon_j)=\sum_{t=1}^m b_{jt}\varepsilon_t
$$
Herewith
\begin{equation}\label{Ind-1}
\Ind(f_*(e_i),\tilde{f}_*(\varepsilon_j))=\sum_{s,t=1}^m a_{is}b_{jt}\Ind(e_s,\varepsilon_t)=
\sum_{s,t=1}^m a_{is}b_{jt}\delta_{st}=\sum_{s=1}^m a_{is}b_{js}.
\end{equation}
On the other hand, since $\deg{f}=\pm1$, then
\begin{equation}\label{Ind-2}
\Ind(f_*(e_i),\tilde{f}_*(\varepsilon_j))=\pm\Ind(e_i,\varepsilon_j)=\pm\delta_{ij}.
\end{equation}
\eqref{Ind-1} and \eqref{Ind-2} imply $B^T=\pm A^{-1}$. Therefore, the roots of the characteristic equations $|A-\lambda E|=0$ and $|B-\tilde\lambda E|=0$ are related by the equation $\tilde\lambda=\pm\lambda^{-1}$. Thus, the first statement is proved.

For the Poincare-Lefschetz isomorphism $l:H^k(M)\rightarrow H_{n-k}(M,\partial{M})$ the following diagram is commutative
\begin{equation}\label{q-f}
\begin{CD}
H^k(M)                 @<{f^*}<<         H^k(M)\\
@V{\pm l}VV                                             @VV{l}V\\
H_{n-k}(M,\partial{M}) @>{\tilde{f}_*}>> H_{n-k}(M,\partial{M}).
\end{CD}
\end{equation}
Let $v\in H_{n-k}(M,\partial{M})$, $v\ne0$, $\tilde\lambda\in\mathbb{R}$ и $\tilde{f}_*(v)=\tilde\lambda v$. Then $\tilde{f}_*^{-1}(v)=\tilde\lambda^{-1}v$. Set $\alpha=l^{-1}(v)$. Since $l$ is an isomorphism, then $\alpha\ne0$. According to \eqref{q-f} we have
$$
f^*(\alpha)=\pm l^{-1}\circ\tilde{f}_*^{-1}\circ l(\alpha)=\pm l^{-1}\circ\tilde{f}_*^{-1}(v)=
\pm l^{-1}(\tilde\lambda^{-1}v)=\pm\tilde\lambda^{-1}l^{-1}(v)=\pm\tilde\lambda^{-1}\alpha.
$$
Thus, $\lambda=\tilde\lambda^{-1}$ is an eigenvalue of the automorphism $f^*$ corresponding to the eigenvector $\alpha\in H^k(M)$.
\end{proof}

According to the lemma proved above for the eigenvalues $\lambda^*_{k,j},\,j\in\{1,\dots,\dim\,H^k(M)\}$ of $f^{*}_{k}$  and $f^m$ such that $I(p,f^m)=I(q,f^m)$ for any $p,q\in Fix(f^m)$ the following equality takes place  \begin{equation}\label{+Lef+}
N_m=\left|\sum\limits_{k=0}^n(-1)^{n-k}\left(\sum\limits_{j=1}^{\dim\,H^k(M)}\lambda_{k,j}^{*m}\right)\right|.
\end{equation}

\subsection{Filtration}\label{filtr}
Let $f:M^n\to M^n$ be an $\Omega$-stable diffeomorphism. As $f$ has no  cycles then $\prec$ is a partial order relation  on the basic sets
$$\Lambda_i\prec\Lambda_j\iff W^s_{\Lambda_i}\cap W^u_{\Lambda_j}\neq\emptyset.$$
Intuitively the definition means that ``everything trickles down'' towards ``smaller elements''. The partial order $\prec$ extends to the order relation, i.e. the basic sets can be enumerated $\Lambda_1,\dots,\Lambda_m$ in accordance with the relation $\prec$: $$\hbox{if}\,\Lambda_i\prec\Lambda_j,\,\hbox{then}\,i\leq j.$$

We pick a sequence of nested subsets of the ambient manifold $M^n$ in the following way. Let the first subset of  $M^n$ be a neighborhood $M_1$ of the basic set $\Lambda_1$, let the next subset $M_2$ be the union of $M_1$ and some neighborhood of the unstable manifold of the element $\Lambda_2$. If we continue this process we get the entire manifold $M^n$. This construction  gives the idea to the following notion of filtration.

 A sequence $M_1,\dots,M_{m-1}$ of compact $n$-submanifolds of $M^n$, each having a smooth boundary, and such that $M^n=M_m\supset M_{m-1}\supset\dots\supset M_1\supset
M_0=\emptyset$ is called a {\em filtration} for a diffeomorphism $f$ with its ordered basic sets $\Lambda_1\prec\dots\prec\Lambda_m$ if for each $i=1,\dots,m$ the following holds:
\begin{enumerate}
\item $f(M_i)\subset {\rm int}\, M_i$;
\item $\Lambda_i\subset {\rm int}\, (M_i\setminus M_{i-1})$;
\item $\Lambda_i=\bigcap\limits_{l\in\mathbb
Z}f^l(M_i\setminus M_{i-1})$;
\item $\bigcap\limits_{l\geq
0}f^l(M_i)=\bigcup\limits_{j\leq
i}W^u_{\Lambda_j}=\bigcup\limits_{j\leq
i}cl(W^u_{\Lambda_j})$.
\end{enumerate}
Below we describe following from  \cite{Franksdim} interrelations between actions $f$ on cohomology groups $H^k(M^n)$, $H^k(M_i,M_{i-1})$  and homology group $H_k(M^n)$ with real coefficients. If an action in these group is {\it nilpotent} then all eigenvalues equal zero and if it is {\it unipotent} then it has only roots of unity
as eigenvalues.
\begin{proposition}\label{dimFr} Let $f:M^n\to M^n$ be an $\Omega$-stable diffeomorphism and $M^n=M_m\supset M_{m-1}\supset\dots\supset M_1\supset
M_0=\emptyset$ be a filtration for its ordered basic sets $\Lambda_1\prec\dots\prec\Lambda_m$. Then
\begin{enumerate}
\item If $\lambda$ is an eigenvalue of $f^*_k: H^k(M^n)\to H^k(M^n)$, then there is an $i\in\{1,\dots,m\}$  such that $f^*_k: H^k(M_i, M_{i-1})\to H^k(M_i, M_{i-1})$ has $\lambda$ as an eigenvalue.
\item If $\Lambda_i$ is a trivial basic set then  $f^*_k: H^k(M_i, M_{i-1})\to H^k(M_i, M_{i-1})$ is nilpotent unless $k =\dim\, W^u_x,\,x\in\Lambda_i$ and $f^*_k: H^k(M_i, M_{i-1})\to H^k(M_i, M_{i-1})$ is unipotent for $k =\dim\, W^u_x,\,x\in\Lambda_i$.
\end{enumerate}
\end{proposition}

\section{Proof of  theorem \ref{mt}}
In this section we prove that if $f:M^3\to M^3$ is an $\Omega$-stable diffeomorphism whose basic sets  are trivial except attractors, then every non-trivial attractor is either one-dimensional non-orientable or two-dimensional expanding. We will use in this proof some results, which will be proven in the next section.
As above, the symbols $H_k(X,A)$ and $H^k(X,A)$ will denote homology and cohomology groups with real coefficients. For homology groups with integer coefficients, the notation $H_k(X,A;\mathbb{Z})$ will be used.

\begin{proof} Suppose the contrary: $NW(f)$ contains a non-trivial attractor $A$ such that $A$ is either one-dimensional orientable or two-dimensional Anosov torus. Without loss of generality we can assume that in the order $\prec$, first positions occupied by attractors and $A$ is the last of them. Let $M^n=M_k\supset M_{k-1}\supset\dots\supset M_1\supset
M_0=\emptyset$ be a filtration for the ordered basic sets $\Lambda_1\prec\dots\prec\Lambda_k$. Then $\tilde M_i=M^n\setminus {\rm int}\,M_{k-i}$ is the filtration for the basic sets $\tilde\Lambda_i=\Lambda_{k-i}$ of the diffeomorphism $g=f^{-1}$. Let $A=\tilde\Lambda_{i_0}$. Without loss of generality we can assume that the manifold $\tilde M_{i_0}$ is connected (in the opposite case let us consider its connected component containing $A$). Then $g(\tilde M_{i_0})\subset {\rm int}\,\tilde M_{i_0}$. Notice, that $i_0>1$ since any $\Omega$-stable diffeomorphism has non-empty sets of attractors and repellers.

Let $N_m$ be a number of points in $Fix(g^m)$. As the non-trivial basic set $A$ belongs to $\tilde M_{i_0}$ then  $\lim\limits_{m\to\infty}N_m=\infty$. Since $A$ is orientable then  the Lemma \ref{ind} and the formula (\ref{+Lef+}) gives the existence of an eigenvalue $\lambda$ with absolute value greater than 1 for $g^*_{k}:H^k(\tilde M_{i_0})\to H^k(\tilde M_{i_0})$ for some $k\in\{0,\dots,3\}$.

First of all, let us show that it is impossible for orientable $M^3$. We will prove it separately for each dimension $k=0,1,2,3$.

a) $k=0$. Eigenvalues of the automorphism
$g^*: H^0(\tilde M_{i_0})\to H^0(\tilde M_{i_0})$ are roots of unity by the lemma \ref{H^0}.

b) $k=3$. The group $H_3(\tilde M_{i_0};\mathbb{Z})$ is trivial when $\partial\tilde{M}_{i_0}\ne\emptyset$ and is isomorphic to $\mathbb{Z}$ when 
$\partial\tilde M_{i_0}=\emptyset$. In the first case we have
$H^3(\tilde{M}_{i_0})=0$ and so
$g^*: H^3(\tilde M_{i_0})\to H^3(\tilde M_{i_0})$ does not have eigenvalues. In the second case, $g^*=\pm\id$ by the lemma \ref{H=Z}.

c) $k=1$. Suppose, that the automorphism
$g^*:H^1(\tilde M_{i_0})\to H^1(\tilde M_{i_0})$ has an eigenvalue $\lambda$, for which $\lambda^2\ne1$. Then it follows from the item 1 of the proposition \ref{dimFr}, that there exists a number $i$, $1\leqslant{i}\leqslant{i_0}$ such that the automorphism
$g^*:H^1(M_i,M_{i-1})\rightarrow H^1(M_i,M_{i-1})$ also has the eigenvalue $\lambda$.

As all basic sets of $g$ before $A$ in the Smale order $\prec$ are trivial then by the item 2 of proposition \ref{dimFr} for $i<i_0$ we get that the automorphisms $g^*$ on $H^1(M_i,M_{i-1})$ are either nilpotent or uniponent. Hence, it is precisely the automorphism
$g^*:H^1(\tilde M_{i_0},\tilde M_{i_0-1})\rightarrow H^1(\tilde M_{i_0},\tilde M_{i_0-1})$ must have the eigenvalue $\lambda$.

Let $\dim\,A=1$. In this case $\tilde M_{i_0}=Q_g\cup\tilde M_{i_0-1}$, where $Q_g$ is a handlebody of a genus $g\geqslant 0$ such that $Q_g\cap\tilde M_{i_0-1}=\partial Q_g$.
By lemma \ref{1z} $H_1(\tilde M_{i_0},\tilde M_{i_0-1};\mathbb{Z})=0$. Then
$H^1(\tilde M_{i_0},\tilde M_{i_0-1})=0$ and therefore $\lambda$ cannot be an eigenvalue of the automorphism $g^*$.

If $\dim\,A=2$, then $\tilde M_{i_0}=Q\cup\tilde M_{i_0-1}$, where $Q\cong\mathbb T^2\times[0,1]$ and $Q\cap\tilde M_{i_0-1}=\partial Q$. In this situation by the lemma \ref{2z}
$H_1(\tilde M_{i_0},\tilde M_{i_0-1};\mathbb{Z})\cong\mathbb{Z}$. From here and from the lemma \ref{H=Z} it follows, that $g^*=\pm\id$. Thus, we obtain a contradiction for $k=1$ as well.

d) $k=2$. Let us finally assume that
$g^*:H^2(\tilde M_{i_0})\to H^2(\tilde M_{i_0})$ has as eigenvalue $\lambda$, for which $\lambda^2\ne1$.
Due to lemma \ref{lambda_k}, in such a situation the automorphism
$g^*:H^1(\tilde{M}_{i_0},\partial\tilde{M}_{i_0})\to H^1(\tilde{M}_{i_0},\partial\tilde{M}_{i_0})$
has an eigenvalue $\tilde\lambda=\pm\lambda^{-1}$.

Consider the following diagram
\begin{equation}\label{CD-f^*}
\begin{CD}
\dots @>{}>> H^0(\partial\tilde{M}_{i_0}) @>{\delta^*}>> H^1(\tilde{M}_{i_0},\partial\tilde{M}_{i_0}) @>{j^*}>> H^1(\tilde{M}_{i_0}) @>{}>> \dots\\
@.  @VV{g^*}V         @VV{g^*}V        @VV{g^*}V  @.  \\
\dots @>{}>> H^0(\partial\tilde{M}_{i_0}) @>{\delta^*}>> H^1(\tilde{M}_{i_0},\partial\tilde{M}_{i_0}) @>{j^*}>> H^1(\tilde{M}_{i_0}) @>{}>> \dots,
\end{CD}
\end{equation}
where the rows are taken from the cohomological sequence of the pair $(\tilde{M}_{i_0},\partial\tilde{M}_{i_0})$ and the vertical arrows denote the mappings induced by the diffeomorphism $g$.
All squares of the diagram are commutative, and the middle automorphism $g^*$ from \eqref{CD-f^*} has an eigenvalue $\tilde\lambda$. From this, by \cite[Lemma 3]{Franksdim} it follows that for one of the extreme vertical automorphisms of the diagram \eqref{CD-f^*} $\tilde\lambda$ is also an eigenvalue. Since $\tilde\lambda^2\ne1$, then for the automorphism
$g^*:H^1(\tilde M_{i_0})\to H^1(\tilde M_{i_0})$ this is impossible according to proven in c). Since the manifold $\tilde M_{i_0}$ is compact, its boundary $\partial\tilde{M}_{i_0}$ consists of a finite set of connected components. Then by Lemma \ref{H^0} all eigenvalues of the automorphism
$g^*:H^0(\partial\tilde M_{i_0})\to H^0(\partial\tilde M_{i_0})$
are roots of unity. Thus, in this case we also obtain a contradiction.

If $M^n$ is non-orientable then, by lemma \ref{bar f}, there is an oriented two-fold covering $p:\bar M^n\to M^n$ and a lift $\bar g:\bar M^n\to\bar M^n$ of the diffeomorphism\footnote{We have not found a reference for this fact, so we prove it in the section  \ref{sec:two-fold_covering} below.} $g$. Herewith, by lemma \ref{oro},  $\bar A=p^{-1}(A)$ is orientable, like $A$. So we can apply all arguments from an orientable case to $\bar g$ and get a contradiction. \end{proof}

\section{Homology and induced automorphisms}
In this section, we calculate the homology groups of some topological pairs and study the properties of automorphisms of cohomology groups induced by homeomorphisms.

\subsection{Calculations}
In this section we calculate relative homology for the following situation. Let $M$ and $N$ be smooth 3-manifolds with boundaries such that $P=M\cup{N}$ is connected, $M\cap{N} = \partial{M}$ and connected components of $\partial{M}$ are some of connected components of  $\partial{N}$. Let us calculate the relative homology groups of the pair $(P,N)$.

Firstly notice that $H_0(P,N;\mathbb{Z})=0$ as the manifold $P$ is connected and $N$ is not empty. For the calculation of other relative homology groups we need the following fact.
\begin{lemma}\label{reduction}
For every natural $k$ the following isomorphism takes place
$$H_k(P,N;\mathbb{Z})\cong H_k(M,\partial{M};\mathbb{Z}).$$
\end{lemma}
\begin{proof} By \cite[Theorem 6.1, Chapter 4]{Hirsch} the boundary $\partial{N}$ possesses a collar in  $N$. As connected components of $\partial{M}$ are connected components of $\partial{N}$ then there is an embedding $\phi:\partial{M}\times[0,1)\rightarrow{N}$ such that $\phi(a,0)=a$ for every $a\in\partial{M}$. Let  $V=\phi(\partial{M}\times[0,1))$,   $B=N\setminus{V}$ and ${\rm cl}\,{B}$ be the closure of $B\subset{P}$ in $P$,  $ {\rm int}\,{N}$ be the interior of $N\subset{P}$ in $P$.
By the construction ${\rm cl}\,{B}=B\subset{N}\setminus\partial{M}={\rm int}{N}$. The excision theorem \cite[Corollary 7.4, Chapter III]{Dold} claims in such case that
$$H_k(P,N;\mathbb{Z})\cong H_k(P\setminus{B},N\setminus{B};\mathbb{Z})=H_k(M\cup{V},V;\mathbb{Z}).$$
But the pair $(M\cup{V},V)$ is homotopically equivalent to the pair $(M,\partial{M})$.
Hence, $H_k(M\cup{V},V;\mathbb{Z})\cong H_k(M,\partial{M};\mathbb{Z})$ for a natural $k$.
\end{proof}
Below we calculate $H_k(M,\partial{M};\mathbb{Z})$ in two cases: 1) $M$ is a handlebody of a genus $g\geqslant 0$, 2) $M\cong\mathbb T^2\times[0,1]$.
\begin{lemma}\label{1z} If $M$ is a handlebody of a genus $g\geqslant 0$ then
\begin{equation}
H_3(M,\partial{M};\mathbb{Z})\cong\mathbb{Z},\quad H_2(M,\partial{M};\mathbb{Z})\cong\mathbb{Z}^g,\quad
H_1(M,\partial{M};\mathbb{Z})=0.
\end{equation}
\end{lemma}
\begin{proof}
As $H_3(M;\mathbb{Z})=0$ and  $H_0(M,\partial{M};\mathbb{Z})=0$ then the homological sequence of the pair
$(M,\partial M)$  has the following form \cite[Proposition 4.4, Chapter III]{Dold}:
\begin{multline}\label{pair sequence}
0\longrightarrow H_3(M,\partial{M};\mathbb{Z})\stackrel{\partial_*^3}{\longrightarrow}
H_2(\partial{M};\mathbb{Z})\stackrel{\imath_*^2}{\longrightarrow}
H_2(M;\mathbb{Z})\stackrel{\jmath_*^2}{\longrightarrow}
H_2(M,\partial{M};\mathbb{Z})\stackrel{\partial_*^2}{\longrightarrow}\\
{\longrightarrow}
H_1(\partial{M};\mathbb{Z})\stackrel{\imath_*^1}{\longrightarrow}
H_1(M;\mathbb{Z})\stackrel{\jmath_*^1}{\longrightarrow}
H_1(M,\partial{M};\mathbb{Z})\stackrel{\partial_*^1}{\longrightarrow}\\
{\longrightarrow}
H_0(\partial{M};\mathbb{Z})\stackrel{\imath_*^0}{\longrightarrow}
H_0(M;\mathbb{Z})\stackrel{\jmath_*^0}{\longrightarrow}0.
\end{multline}

Handlebody $M$ of a genus $g$ is the 3-ball with glued  $g$ 3-handles of the index 1. That is $M$ is homotopically equivalent to the bouquet of $g$ circles.
Therefore,
$$
H_2(M;\mathbb{Z})=0,\quad H_1(M;\mathbb{Z})\cong\mathbb{Z}^g,\quad H_0(M;\mathbb{Z})\cong\mathbb{Z}.
$$
On the other side the boundary
 $\partial{M}$ is homeomorphic to the surface $S_g$ of the genus  $g$.
Hence,
$$
H_2(\partial M;\mathbb{Z})\cong\mathbb{Z},\quad H_1(\partial M;\mathbb{Z})\cong\mathbb{Z}^{2g},\quad
H_0(\partial M;\mathbb{Z})\cong\mathbb{Z}.
$$
Substituting the latter in \eqref{pair sequence}, we get the exact sequence
\begin{multline}\label{pair sequence for k>0}
0\longrightarrow H_3(M,\partial M;\mathbb{Z})\stackrel{\partial_*^3}{\longrightarrow}
\mathbb{Z}\stackrel{\imath_*^2}{\longrightarrow}
0\stackrel{\jmath_*^2}{\longrightarrow}
H_2(M,\partial M;\mathbb{Z})\stackrel{\partial_*^2}{\longrightarrow}\\
{\longrightarrow}
\mathbb{Z}^{2g}\stackrel{\imath_*^1}{\longrightarrow}
\mathbb{Z}^g\stackrel{\jmath_*^1}{\longrightarrow}
H_1(M,\partial M;\mathbb{Z})\stackrel{\partial_*^1}{\longrightarrow}
\mathbb{Z}\stackrel{\imath_*^0}{\longrightarrow}
\mathbb{Z}\stackrel{\jmath_*^0}{\longrightarrow}0.
\end{multline}
As $\imath_*^1$ is an epimorphism then \eqref{pair sequence for k>0} decomposes into short exact sequences
\begin{align*}
0& \longrightarrow H_3(M,\partial M;\mathbb{Z})\stackrel{\partial_*^3}{\longrightarrow}
\mathbb{Z}\stackrel{\imath_*^2}{\longrightarrow}0,\\
0& \longrightarrow H_2(M,\partial M;\mathbb{Z})\stackrel{\partial_*^2}{\longrightarrow}
\mathbb{Z}^{2g}\stackrel{\imath_*^1}{\longrightarrow}
\mathbb{Z}^g\longrightarrow 0,\\
0& \longrightarrow H_1(M,\partial M;\mathbb{Z})\stackrel{\partial_*^1}{\longrightarrow}
\mathbb{Z}\stackrel{\imath_*^0}{\longrightarrow}
\mathbb{Z}\longrightarrow 0,
\end{align*}
from which follows the statement of the lemma.
\end{proof}
\begin{lemma}\label{2z} If $M=\mathbb T^2\times[0,1]$ then
\begin{equation}
H_3(M,\partial{M};\mathbb{Z})\cong\mathbb{Z},\quad H_2(M,\partial{M};\mathbb{Z})\cong\mathbb{Z}^2,\quad
H_1(M,\partial{M};\mathbb{Z})=\mathbb Z.
\end{equation}
\end{lemma}
\begin{proof}
As $M$ is homotopically equivalent to $\mathbb{T}^2$ and $\partial{M}$ is homeomorphic to $\mathbb{T}^2\times\mathbb S^0$ then $H_k(M;\mathbb{Z})\cong{H}_k(\mathbb T^2;\mathbb{Z})$ and $H_k(\partial{M};\mathbb{Z})\cong{H}_k(\mathbb T^2;\mathbb{Z})\times{H}_k(\mathbb T^2;\mathbb{Z})$.
In such situation the homological sequence \eqref{pair sequence} of the pair $(M,\partial M)$  has the following form:
\begin{multline}\label{pair sequence Z}
0\longrightarrow H_3(M,\partial{M};\mathbb{Z})\stackrel{\partial_*^3}{\longrightarrow}
\mathbb{Z}^2\stackrel{\imath_*^2}{\longrightarrow}
\mathbb{Z}\stackrel{\jmath_*^2}{\longrightarrow}
H_2(M,\partial{M};\mathbb{Z})\stackrel{\partial_*^2}{\longrightarrow}\\
{\longrightarrow}
\mathbb{Z}^4\stackrel{\imath_*^1}{\longrightarrow}
\mathbb{Z}^2\stackrel{\jmath_*^1}{\longrightarrow}
H_1(M,\partial{M};\mathbb{Z})\stackrel{\partial_*^1}{\longrightarrow}
\mathbb{Z}^2\stackrel{\imath_*^0}{\longrightarrow}
\mathbb{Z}\stackrel{\jmath_*^0}{\longrightarrow}0.
\end{multline}
As the inclusion of every connected component of $\partial M$ to ${M}$ is a homotopical equivalence then $\imath_*^2$ and $\imath_*^1$ are epimorphisms. Herewith \eqref{pair sequence Z} decomposes into short exact sequences
\begin{align*}
0& \longrightarrow H_3(M,\partial{M};\mathbb{Z})\stackrel{\partial_*^3}{\longrightarrow}
\mathbb{Z}^2\stackrel{\imath_*^2}{\longrightarrow}
\mathbb{Z}\longrightarrow 0,\\
0& \longrightarrow H_2(M,\partial{M};\mathbb{Z})\stackrel{\partial_*^2}{\longrightarrow}
\mathbb{Z}^4\stackrel{\imath_*^1}{\longrightarrow}
\mathbb{Z}^2\longrightarrow 0,\\
0& \longrightarrow H_1(M,\partial{M};\mathbb{Z})\stackrel{\partial_*^1}{\longrightarrow}
\mathbb{Z}^2\stackrel{\imath_*^0}{\longrightarrow}
\mathbb{Z}\longrightarrow 0,
\end{align*}
from which follows the statement of the lemma.
\end{proof}

\subsection{Eigenvalues of induced automorphisms}
In this section we again consider all homology groups $H_k(X,A;\mathbb{Z})$ with integer coefficients and cohomology groups $H^k(X,A)$ with real  coefficients. Firstly, by the  Universal Coefficient Formula \cite[Section 7, Chapter VI]{Dold}, the previous subsection results give the following calculations.
\begin{statement}\label{1} If $M$ is a handlebody of a genus $g\geqslant 0$ then $$H^3(M,\partial{M})\cong\mathbb{R},\, H^2(M,\partial{M})\cong\mathbb{R}^g,\,
H^1(M,\partial{M})=0,\,
H^0(M,\partial{M})=0.$$
\end{statement}

\begin{statement}\label{2} If $M=\mathbb T^2\times[0,1]$ then
$$H^3(M,\partial{M})\cong\mathbb{R},\, H^2(M,\partial{M})\cong\mathbb{R}^2,\,H^1(M,\partial{M})=\mathbb R,\,
H^0(M,\partial{M})=0.$$
\end{statement}

The groups $H^k(X,A)\cong\mathbb{R}^m$ admit many automorphisms even for $m=1$. But in some cases only a small part of them can be induced by homeomorphisms of the topological space $X$.

\begin{lemma}\label{H=Z}
Let $X$ be a topological space, $A\subset{X}$ be its subspace, $f:X\rightarrow{X}$ be homeomorphism, and $f(A)\subset{A}$. Denote by $H'_k(X,A;\mathbb{Z})$ a free part of the group of $k$-dimensional singular homology of the pair $(X,A)$, and by  $H^k(X,A)$ its $k$-dimensional cohomology group with real coefficients.
It $H'_k(X,A;\mathbb{Z})\cong\mathbb{Z}$ for some $k$, then for the induced automorphism $f^*:H^k(X,A)\rightarrow H^k(X,A)$ the equality $f^*=\pm\id$ holds.
\end{lemma}
\begin{proof}
Let the automorphism $f_*:H'_k(X,A;\mathbb{Z})\rightarrow H'_k(X,A;\mathbb{Z})$ also induced by the homeomorphism $f$. The formula $f^*_h(q)=q\circ f_*$ defines the automorphism $f^*_h:\Hom(H'_k(X,A;\mathbb{Z});\mathbb{R})\rightarrow\Hom(H'_k(X,A;\mathbb{Z});\mathbb{R})$.
If $H'_k(X,A;\mathbb{Z})\cong\mathbb{Z}$, then $f_*=\pm\id$. Moreover, $f^*_h(q)=q\circ(\pm\id)=\pm{q}$ for all $q\in\Hom(H'_k(X,A;\mathbb{Z});\mathbb{R})$. Hence $f^*_h=\pm\id$.

It follows for the Universal Coefficient Formula for cohomology \cite[Chapter~VI, Section~7]{Dold} that there exists the natural isomorphism $\kappa:H^k(X,A)\rightarrow\Hom(H'_k(X,A;\mathbb{Z});\mathbb{R})$.
The naturalness means commutativity of the diagram
\begin{equation}\label{CD-kappa}
\begin{CD}
H^k(X,A) @>{\kappa}>> \Hom(H'_k(X,A;\mathbb{Z});\mathbb{R})\\
@V{f^*}VV                        @VV{f^*_h}V\\
H^k(X,A) @>{\kappa}>> \Hom(H'_k(X,A;\mathbb{Z});\mathbb{R}).
\end{CD}
\end{equation}
It follows from \eqref{CD-kappa} and the equation $f^*_h=\pm\id$ that
$f^*=\kappa^{-1}\circ f^*_h\circ\kappa=\kappa^{-1}\circ(\pm\id)\circ\kappa=\pm\id$.
\end{proof}

\begin{lemma}\label{H^0}
Let $X$ be a topological space with a finite number of path-connected components, $f:X\rightarrow{X}$ be a homeomorphism, and $f^*:H^0(X)\rightarrow H^0(X)$ be an induced automorphism. Then any eigenvalue $\lambda$ for $f^*$ satisfies the equality $\lambda^2=1$.
\end{lemma}
\begin{proof}
Firstly consider a case when $X_1$ and $X_2$ are path-connected topological spaces and $f_2:X_1\rightarrow{X}_2$ is a homeomorphism.
All elements of groups $H^0(X_j)$ are constant functions $c_j:X_j\rightarrow\mathbb{R}$. Therefore, the formula $\nu_j(c_j)=\im{c}_j$ defines isomorphisms $\nu_j:H^0(X_j)\rightarrow\mathbb{R}$, $j=1,2$.
The induced isomorphism
$f_2^*:H^0(X_2)\rightarrow{H}^0(X_1)$ is defined by the formula $f_2^*(c_2)=c_2\circ{f}_2$. Since values of the functions $c_2$ and $c_2\circ{f}_2$ are equal, then
\begin{equation}\label{im-nu}
\nu_1\circ{f}^*_2=\nu_2.
\end{equation}

Now suppose that $X$ consists of path-connected components $X_1,\dots,X_m$.
Then there exists a permutation $\sigma\in{S}_m$ such that $f$ maps the component $X_j$ onto the component $X_{\sigma(j)}$ homeomorphically. Thus, setting $f_{\sigma(j)}(x)=f(x)$ for all $x\in{X}_j$, we obtain homeomorphisms $f_{\sigma(j)}:X_j\rightarrow{X}_{\sigma(j)}$, $j=1,\dots,m$. Moreover, the induced homomorphisms $f_j^*:H^0(X_j)\rightarrow{H}^0(X_{\tau(j)})$ are defined, where $\tau=\sigma^{-1}$.
By virtue of \eqref{im-nu}
\begin{equation}\label{im-nu-j}
\nu_{\tau(j)}\circ{f}_j^*=\nu_j,\quad j=1,\dots,m.
\end{equation}

For each element $c\in{H}^0(X)$ we set $c_j=c|_{X_j}$. Then $c_j\in{H}^0(X_j)$. Define isomorphisms $\mu:{H}^0(X)\rightarrow{H}^0(X_1)\times\dots\times{H}^0(X_m)$ and $\nu:{H}^0(X_1)\times\dots\times{H}^0(X_m)\rightarrow\mathbb{R}^m$ by the formulas $\mu(c)=(c_1,\dots,c_m)$ and $\nu((c_1,\dots,c_m))=(\nu_1(c_1),\dots,\nu_m(c_m))$. We construct the automorphism $p:\mathbb{R}^m\rightarrow\mathbb{R}^m$ such that the diagram is commutative
\begin{equation}\label{CD-p}
\begin{CD}
H^0(X) @>{\mu}>>
H^0(X_1)\times\dots\times H^0(X_m)  @>{\nu}>> \mathbb{R}^m\\
@VV{f^*}V          @VV{(f^*_1,\dots,f^*_m)}V              @VV{p}V\\
H^0(X) @>{\mu}>>
H^0(X_1)\times\dots\times H^0(X_m)  @>{\nu}>> \mathbb{R}^m.
\end{CD}
\end{equation}

For all $y=(y_1,\dots,y_m)\in\mathbb{R}^m$ we set $\|y\|=\sqrt{y_1^2+\dots+y_m^2}$.
Since $f_j^*$ maps $H^0(X_j)$ onto ${H}^0(X_{\tau(j)})$, then it follows from the equality \eqref{im-nu-j} and the diagram \eqref{CD-p} that $p(y)=(y_{\tau(1)},\dots,y_{\tau(1)})$.
Moreover, $\|{p}(y)\|=\|y\|$.

Finally, let $\lambda\in\mathbb{R}$, $c\in{H}^0(X)$, $c\ne0$ and $f^*(c)=\lambda{c}$. We set $y=\nu\circ\mu(c)$. Then by virtue of \eqref{CD-p} $p(y)=\mu\circ\nu(f^*(c))=\mu\circ\nu(\lambda{c})=\lambda{y}$.
Hence, according to what was proved above, we obtain $\|y\|^2=\|{p}(y)\|^2=\lambda^2\|y\|^2$. Hence, $\lambda^2=1$.
\end{proof}

\section{On oriented two-fold covering}\label{sec:two-fold_covering}
Let $M$ be a non-orientable connected smooth $n$-manifold, $a\in M$ and $x:I\rightarrow M$ be a loop based at a point  $a$. Let us consider continuous vector fields $X_1,\dots,X_n$ along $x$ such that $X_1(t),\dots,X_n(t)$ linearly independent for each $t\in{I}$. Then there is a matrix  $A=(a_i^j)\in{\rm GL}_n(\mathbb{R})$ such that
\begin{equation}\label{matrix A}
X_i(1)=a_i^jX_j(0),\quad i,j=1,\dots,n.
\end{equation}
Let $\omega_a(x)={\rm sign}\det{A}$. If $y$ is a loop which based at the same starting point and $x\sim{y}$ then $\omega_a(x)=\omega_a(y)$. Therefore the formula $\omega_a([x])=\omega_a(x)$ defines a homeomorphism $\omega_a:\pi_1(M,a)\rightarrow{G}$, where  $G=\{1,-1\}$. The manifold $M$ is orientable if and only if $\ker\omega_a=\pi_1(M,a)$.

Let $a,b\in{M}$, $z:I\rightarrow M$ be a path which starts in  $z(0)=a$ and ends in $z(1)=b$ and $T_z:\pi_1(M,a)\rightarrow\pi_1(M,b)$ be the isomorphism defined by the formula $T_z([x])=[z^{-1}xz]$. Then  $zz^{-1}\sim 1_a$ and  $z^{-1}z\sim 1_b$ implies commutativity of the diagram
\begin{equation}\label{CD-omega-Tz}
\begin{CD}
\pi_1(M,a) @>{\omega_a}>> \mathbb{R}\\
@V{T_z}VV                 @VV{{\rm id}}V\\
\pi_1(M,b) @>{\omega_b}>> \mathbb{R}.
\end{CD}
\end{equation}

\begin{lemma}\label{omega--f}
Let $M,N$ be connected smooth manifolds, $f:M\rightarrow N$ be a local diffeomorphism,
$a\in M$, $b=f(a)$ and $f_*:\pi_1(M,a)\rightarrow\pi_1(N,b)$ be an induced homeomorphism. Then the following diagram is commutative
\begin{equation}\label{CD-omega--f}
\begin{CD}
\pi_1(M,a) @>{\omega_a}>> \mathbb{R}\\
@V{f_*}VV                 @VV{{\rm id}}V\\
\pi_1(N,b) @>{\omega_b}>> \mathbb{R}.
\end{CD}
\end{equation}
\end{lemma}
\begin{proof} Let $[x]\in\pi_1(M,a)$, $X_1,\dots,X_n$ be  continuous vector fields along $x$, linearly independent at each point $x(t)$, and the equality \eqref{matrix A} is satisfied. Let  $y=f\circ{x}$ and $Y_i(t)=df_{x(t)}(X_i(t))$ for every $i=1,\dots,n$ and $t\in{I}$.
Then $[y]\in\pi_1(N,b)$, $[y]=f_*([x])$ and $Y_1,\dots,Y_n$ are continuous vector fields along the loop $y$. According to the condition,  $df_{x(t)}:T_{x(t)}M\rightarrow T_{y(t)}N$ are  isomorphisms. Therefore $Y_1(t),\dots,Y_n(t)$ are linearly dependent for all  $t\in{I}$.
But
$$
Y_i(1)=df_a(X_i(1))=df_a(a_i^jX_j(0))=a_i^jdf_a(X_j(0))=a_i^jY_j(0).
$$ from \eqref{matrix A} and the linearity of the differential $df_a:T_aM\rightarrow T_bN$. Thus, $\omega_b([y])={\rm sign}\det{A}=\omega_a([x])$.
\end{proof}

\begin{lemma}\label{bar f}
Let $M$ be a non-orientable connected smooth manifold and $f:M\rightarrow M$ be a diffeomorphism. Then there exists a connected smooth orientable manifold $\bar{M}$, a smooth two-fold cover $p:\bar{M}\rightarrow M$ and a diffeomorphism  $\bar{f}:\bar{M}\rightarrow\bar{M}$ for which the diagram is commutative
\begin{equation}\label{CD-covering-diffeomorphism}
\begin{CD}
\bar{M}           @>{\bar{f}}>> \bar{M}\\
@V{p}VV               @VV{p}V\\
{M}               @>{f}>>       {M}.
\end{CD}
\end{equation}
\end{lemma}
\begin{proof} Let $a\in M$. Then ${\rm ker}\,\omega_a$ is the normal divisor of the group $\pi_1(M,a)$. By the theorem of the existence of covers, there will be a connected smooth manifold $\bar{M}$, a regular smooth cover $p:\bar{M}\rightarrow M$ and a point $u\in\bar{M}$ such that $p(u)=a$ and the induced homomorphism  $p_*^u:\pi_1(\bar{M},u)\rightarrow\pi_1(M,a)$ has the image ${\rm im}\,{p}_*^u={\rm ker}\,\omega_a$. As the manifold $M$ is non-orientable then $\pi_1(M,a)/{\rm ker}\,\omega_a\cong{G}$. Therefore  $p$ is a two-fold covering. As  $p_*^u:\pi_1(\bar{M},u)\rightarrow\ker\omega_a$ is an isomorphism then by Lemma \ref{omega--f} we get  $\ker\omega_u=\pi_1(\bar{M},u)$.
That means $\bar{M}$ is an orientable manifold.

Let $b=f(a)$ and $v\in{p}^{-1}(b)$. As the manifold $\bar{M}$ is connected then there is a path $\bar{z}:I\rightarrow\bar{M}$ with the starting in $\bar{z}(0)=u$ and the end in $\bar{z}(1)=v$. Let $z=p\circ\bar{z}$. Then $z(0)=a$, $z(1)=b$ and
\begin{equation}\label{impuimpv}
{\rm im}\,{p}_*^v=T_z({\rm im}\,{p}_*^u).
\end{equation}
As $T_z:\pi_1(M,a)\rightarrow\pi_1(M,b)$ is an isomorphism then \eqref{CD-omega-Tz} implies
\begin{equation}\label{Tz-ker-omegaa}
\ker\omega_b=T_z(\ker\omega_a).
\end{equation}
Finitely, as $f_*:\pi_1(M,a)\rightarrow\pi_1(M,b)$  is an isomorphism, \eqref{CD-omega--f} implies the equality
\begin{equation}\label{f-ker-omegaa}
\ker\omega_b=f_*(\ker\omega_a).
\end{equation}
If follows from \eqref{impuimpv}, \eqref{Tz-ker-omegaa}, \eqref{f-ker-omegaa} and the equality ${\rm im}\,{p}_*^u={\rm ker}\,\omega_a$ that
$$
{\rm im}\,(f\circ{p})_*^u=f_*({\rm im}\,{p}_*^u)=f_*({\rm ker}\,\omega_a)={\rm ker}\,\omega_b=
T_z({\rm ker}\,\omega_a)=T_z({\rm im}\,{p}_*^u)={\rm im}\,{p}_*^v.
$$

According to a theorem from the theory of covering, in such a situation there is a map $\bar{f}:\bar{M}\rightarrow\bar{M}$ such that $\bar{f}(u)=v$ and the diagram \eqref{CD-covering-diffeomorphism} is commutative.
This mapping is uniquely defined and is smooth.
Similarly, it is proved that for the inverse diffeomorphism $f^{-1}:M\rightarrow{M}$ there is a smooth map $\overline{f^{-1}}:\bar{M}\rightarrow\bar{M}$ such that $\overline{f^{-1}}(v)=u$ and the following diagram  is commutative
\begin{equation}\label{CD-covering-diffeomorphism-1}
\begin{CD}
\bar{M}           @>{\overline{f^{-1}}}>> \bar{M}\\
@V{p}VV               @VV{p}V\\
{M}               @>{f^{-1}}>>       {M}.
\end{CD}
\end{equation}
Adding \eqref{CD-covering-diffeomorphism-1} to \eqref{CD-covering-diffeomorphism} on the right and on the left, we get the equality $\overline{f^{-1}}\circ\bar{f}={\rm id}$ and $\bar{f}\circ\overline{f^{-1}}={\rm id}$. Therefore $\overline{f^{-1}}=\bar{f}^{-1}$ and $\bar{f}$ is a diffeomorphism.
\end{proof}

\begin{lemma}\label{oro} Let $M$ be a smooth closed non-orientable connected 3-manifold and $W^1,\,W^2\subset M$ be immersions of open balls $D^1,\,D^2$ accordingly, such that $\Ind_{x}(W^1,W^2)=\Ind_{y}({W}^1,{W}^2)$ for every points $x,\,y\in ( W^1\cap W^2)$. If $p:\bar{M}\rightarrow{M}$ is an oriented double covering then $\bar W^1=p^{-1}(W^1),\,\bar W^2=p^{-1}(W^2)$ be immersions of two copies  of open balls $D^1,\,D^2$ accordingly, $\bar W^1=\bar W^1_1\sqcup \bar W^1_2$, $\bar W^2=\bar W^2_1\sqcup \bar W^2_2$,  and  $\Ind_{\bar x}(\bar W^1_i,\bar W^2_j)=\Ind_{\bar{y}}(\bar{W}^1_i,\bar{W}^2_j)$ for every points $\bar x,\,\bar y\in (\bar W^1_i\cap \bar W^2_j)$, $i,\,j=1,2$.
\end{lemma}
\begin{proof} Consider a tubular neighborhood $U^k$ of the submanifolds $W^k$. Since the open subsets $U^k\subset{M}$, $k=1,2$, are contractible, they are regular  covered neighborhoods. That is  $p^{-1}(U^k)=\bar{U}^k_1\cup\bar{U}^k_2$, where $\bar{U}^k_1\cap\bar{U}^k_2=\emptyset$ and  $p|_{\bar{U}^k_i}:\bar{U}^k_i\rightarrow{U^k}$ are diffeomorphisms, $i=1,2$.
Then the sets $\bar{U}^k_i$ are tubular neighborhoods of smooth submanifolds $\bar{W}^k_i\subset\bar{M}$, and the differences $\bar{U}^2_i\setminus\bar{W}^2_i$ consist of the connected components $\bar{U}^2_{i+}$ и $\bar{U}^2_{i-}$.

Let $\bar\sigma_i:\bar{U}^2_{i+}\cup\bar{U}^2_{i-}\rightarrow\mathbb{Z}$ be a function such that $\bar\sigma(\bar{x})=1$ for $\bar{x}\in\bar{U}^2_{i+}$ and $\bar\sigma(\bar{x})=0$ for $\bar{x}\in\bar{U}^2_{i-}$.
As $\bar W^1_i=(p|_{\bar W^1_i})^{-1}(J^1(D^1))$ then
the intersection index in $\bar{x}\in(\bar W^1_i\cap\bar W^2_j)$ is equal to $\Ind_{\bar{x}}(\bar{W}^1_i,\bar{W}^2_j)=\bar\sigma(t+\delta)-\bar\sigma(t-\delta)$, where $\delta$ is a small enough positive number.
Then $\Ind_x(W^1,W^2)=\Ind_{\bar{x}}(\bar{W}^1_i,\bar{W}^2_j)$ and $\Ind_y(W^1,W^2)=\Ind_{\bar{y}}(\bar{W}^1_i,\bar{W}^2_j)$.
So if $\Ind_{x}(W^1,W^2)=\Ind_{y}({W}^1,{W}^2)$ for every points $x,\,y\in ( W^1\cap W^2)$
then $\Ind_{\bar x}(\bar W^1_i,\bar W^2_j)=\Ind_{\bar{y}}(\bar{W}^1_i,\bar{W}^2_j)$ for every points $\bar x,\,\bar y\in (\bar W^1_i\cap \bar W^2_j)$, $i,\,j=1,2$.
\end{proof}

\section{Example of a diffeomorphism with a non-orientable expanding 2-dimensional attractor}\label{sec:example}

Let us construct an example of an $\Omega$-stable diffeomorphism of a closed connected 3-manifold $M^3$ the non-wandering set of which consists of trivial sources, saddles, and a non-orientable expanding 2-dimensional attractor $\Lambda$.

We will start with hyperbolic toral automorphism $L_A:\mathbb{T}^3\to\mathbb{T}^3$ induced by linear map of $\mathbb{R}^3$ with a hyperbolic matrix $A\in GL(3,\mathbb{Z})$, eigenvalues $\lambda_1$, $\lambda_2$, $\lambda_3$ of which such that $0<\lambda_1<1<\lambda_2\leqslant\lambda_3$. The involution $J:\mathbb{T}^3\to\mathbb{T}^3$ defined by the formula $J(x)=-x\pmod{1}$ has 8 fixed points in the 3-torus of the form $(a,b,c)$, where $a,b,c\in\{0,\frac12\}$. Notice that these points are also fixed for $L_A^k$ for some $k\in\mathbb{N}$. Let us ``blow up'' these points like to the classical Smale surgery and such that the surgery commutes with the involution. We will obtain generalized DA-diffeomorphism $f_{GDA}:\mathbb{T}^3\to\mathbb{T}^3$ with 8 fixed sources $\alpha_i$, $i\in\{1,2,\ldots,8\}$ and one 2-dimensional expanding attractor obtained from the diffeomorphism $L_A^k$.

After that we will remove all sources and factorize the basin of the attractor to obtain a new manifold $\tilde M$, i.e. $\tilde M=(\mathbb{T}^3\setminus\bigcup\limits^8_{i=1}{\alpha_i})/_{x\sim -x}$. The natural projection $p:\mathbb{T}^3\setminus\bigcup\limits^8_{i=1}{\alpha_i}\to \tilde M$ is a 2-fold cover. As $f_{GDA}J=Jf_{GDA}$ then $f_{GDA}$ is projected to $\tilde M$ by the diffeomorphism $\tilde f=pf_{GDA}p^{-1}:\tilde M\to\tilde M$ with one 2-dimensional expanding attractor $\Lambda$ and $\tilde M$ is its basin.
The set $\tilde M\setminus \Lambda$ consists of 8 connected components $\tilde N_i$ each of which is diffeomorphic to $\mathbb RP^2\times \mathbb{R}$, where $\mathbb RP^2$ is the real projective plane.

To obtain a fundamental domain $\tilde D_{i}$ of $\tilde f|_{\tilde N_i}$ we can consider a local coordinates $(x,y,z):U_{i}\to \mathbb{R}^3$ in a neighborhood $U_i$ of $\alpha_i$ in which the diffeomorphism $f_{GDA}$ has a form $f_{GDA}(x,y,z)=(2x,2y,2z)$. A fundamental domain of $f_{GDA}|_{W^u_{\alpha_i}\setminus{\{\alpha_i\}}}$ is $D_i=\{(x,y,z)\in\mathbb{R}^3\,|\,1\leqslant x^2+y^2+z^2\leqslant 4\}$  and then the desired fundamental domain $\tilde D_{i}=p(D_i)$. By the construction it is homeomorphic to $RP^2\times [0,1]$. The orbit space of $f_{GDA}|_{W^u_{\alpha_i}\setminus{\{\alpha_i\}}}$ is homeomorphic to $S^2\times S^1$ since each orientation preserving diffeomorphism of $S^2$ is homotopic to identity. Then the orbit space ${\tilde N}_i/{\tilde f}$ can be obtained as $S^2\times S^1|_{\tilde J}$, where $\tilde J$ is involution of $S^2\times S^1$ induced by $J$. Since ${\tilde N}_i/{\tilde f}$ is non-orientable, it follows from \cite{JahrenKwasik2011} that ${\tilde N}_i/{\tilde f}$ is either $S^2\tilde{\times} S^1$, $RP^2\times S^1$, or $RP^3\# RP^3$. The orbit space  ${\tilde N}_i/{\tilde f}$ can also be obtained from the fundamental domain $\tilde{D}_i$ as a mapping torus $RP^2\times [0,1]|_{(x,0)\sim (\tilde{f}(x),1)}$. Hence a fundamental group of the orbit space $\pi_1({\tilde N}_i/{\tilde f})=\mathbb{Z}_2\rtimes_{\tilde{f}}\mathbb{Z}$ and  then it can be only $RP^2\times S^1$.

Consider a gradient-like diffeomorphism $g_1:\mathbb RP^2\to\mathbb RP^2$ with exactly 3 fixed points: a source $\alpha$, a sink $\omega$ and a saddle $\sigma$ (see Fig. \ref{RP2}).
\begin{figure}[h!]
\centerline{\includegraphics
[width=6 true cm]{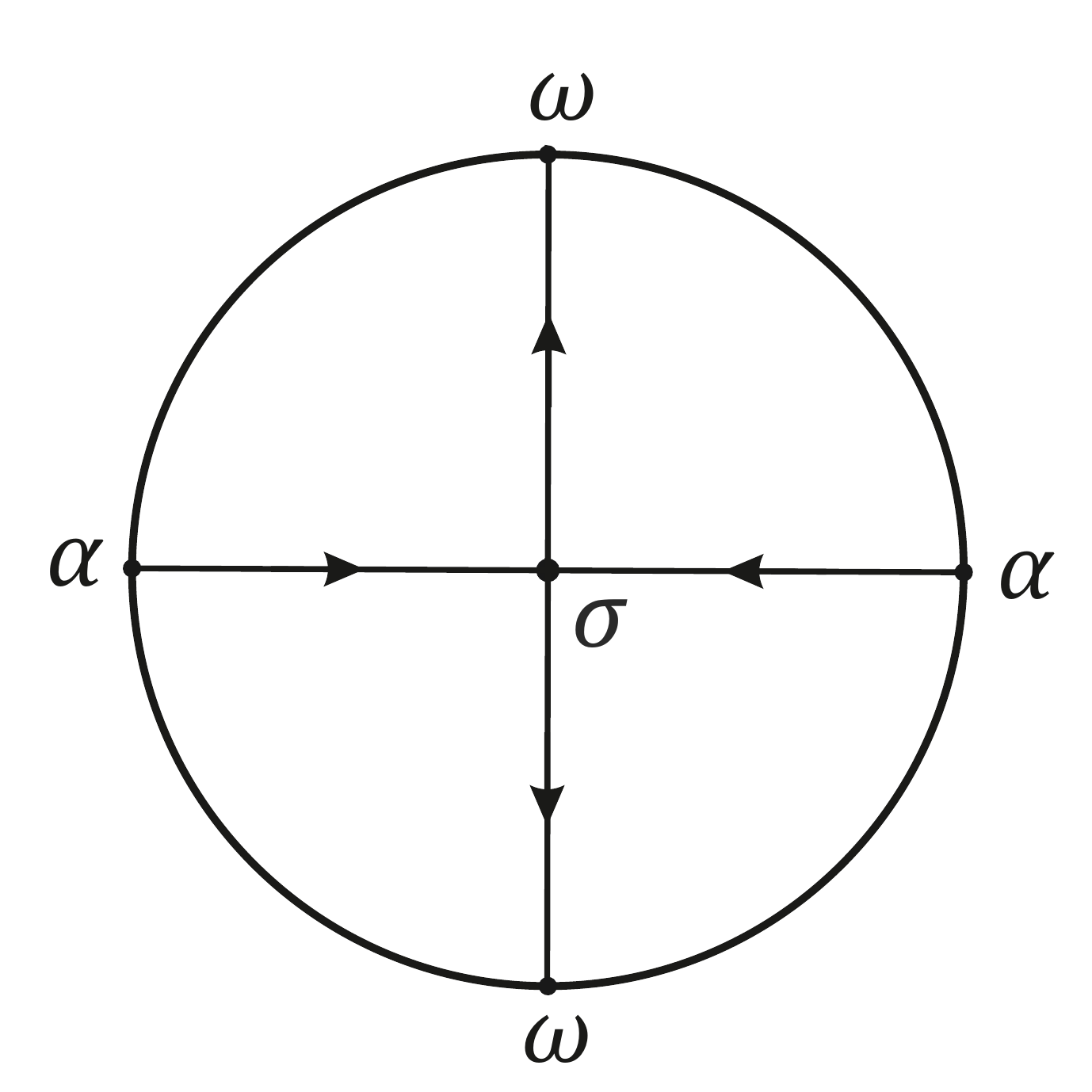}}
\caption{\small Diffeomorphism $g$  on the projective plane}\label{RP2}
\end{figure}
Let $g_2:\mathbb R\to\mathbb R$ be a diffeomorphism given by the formula $g_{2}(x)=2x$ and $g(w,x)=(g_1(w),g_2(x)):\mathbb RP^2\times\mathbb R\to\mathbb RP^2\times\mathbb R$.  Let us denote $N_1,N_2$ the connected components of $\mathbb RP^2\times(\mathbb{R}\setminus\{0\})$.  Analogically with cases with $\tilde N_i$ the orbit spaces  ${N_j}/g$ are diffeomorphic to $\mathbb RP^2\times\mathbb S^1$.

As ${\tilde N_i}/\tilde f$ are diffeomorphic to ${N_j}/g$ then there is a diffeomorphism $h:\tilde N_i\to N_j$ conjugating $\tilde f$ with $g$. Let $h_i:\tilde N_i\to N_1,\,i=1,3,5,7$ and $h_i:\tilde N_i\to N_2,\,i=2,4,6,8$ be such diffeomorphisms. For $\tilde N=\bigcup\limits_{i=1}^8\tilde N_i$ denote by $h:\tilde N\to (N_1\sqcup N_2)\times\mathbb Z_4$ a diffeomorphism composed by $h_i,\,i\in\{1,\dots,8\}$. Let $\tilde P=\mathbb RP^2\times\mathbb R\times\mathbb Z_4$ and $G:\tilde P\to\tilde P$ be a diffeomorphism composed by $g$ on every copy of $\mathbb RP^2\times\mathbb R$. Finitely, let $M^3=\tilde M\cup_h \tilde P$.
Denote by $q:\tilde M\sqcup \tilde P\to M^3$ the natural projection. Then the desired diffeomorphism $ f: M^3 \to M^3 $ coincides with the diffeomorphism $q \tilde f q^{- 1} | _{q (\tilde M)} $ on $q (\tilde M) $ and with the diffeomorphism $q G q^{- 1} | _{q (\tilde P)} $ on $q (\tilde P)$.
\bibliographystyle{ieeetr}
\bibliography{biblio}

\providecommand{\noopsort}[1]{}\providecommand{\singleletter}[1]{#1}%
\begin{thebibliography}{10}

\bibitem{Smale1967}
S.~Smale, ``Differentiable dynamical systems,'' {\em Bull. Amer. Math. Soc.},
  vol.~73, pp.~747--817, 11 1967.

\bibitem{Bowen}
R.~Bowen, ``Periodic points and measures for axiom a diffeomorphisms,'' {\em
  Transactions of the American Mathematical Society}, vol.~154, pp.~377--397,
  1971.

\bibitem{Robinson1999}
C.~Robinson, {\em Dynamical Systems: Stability, Symbolic Dynamics, and Chaos}.
\newblock Studies in Advanced Mathematics, CRC-Press, 1999.

\bibitem{Hirsch}
M.~W. Hirsch, {\em Differential topology}, vol.~33.
\newblock Springer Science \& Business Media, 2012.

\bibitem{Grines1975}
V.~Grines, ``On topological conjugacy of diffeomorphisms of a two-dimensional
  manifold onto one-dimensional orientable basic sets i,'' {\em Transactions of
  the Moscow Mathematical Society}, vol.~32, pp.~31--56, 1975.

\bibitem{Williams74}
R.~F. Williams, ``Expanding attractors,'' {\em Publications Math{\'e}matiques
  de l'IH{\'E}S}, vol.~43, pp.~169--203, 1974.

\bibitem{Plykin1974}
R.~Plykin, ``Sources and sinks of a-diffeomorphisms of surfaces,'' {\em
  Mathematics of the USSR-Sbornik}, vol.~23, no.~2, pp.~243--264, 1974.

\bibitem{Newhouse}
S.~E. Newhouse, ``On codimension one anosov diffeomorphisms,'' {\em American
  Journal of Mathematics}, vol.~92, no.~3, pp.~761--770, 1970.

\bibitem{Brown2010}
A.~W. Brown, ``Nonexpanding attractors: Conjugacy to algebraic models and
  classification in 3-manifolds,'' {\em Journal of Modern Dynamics}, vol.~4,
  pp.~517--548, 07 2010.

\bibitem{GrMeZh2005}
V.~Z. {Grines}, V.~S. {Medvedev}, and E.~V. {Zhuzhoma}, ``{On surface
  attractors and repellers in 3-manifolds},'' {\em {Math. Notes}}, vol.~78,
  no.~6, pp.~757--767, 2005.

\bibitem{Plykin1971}
R.~Plykin, ``The topology of basis sets for smale diffeomorphisms,'' {\em Math.
  USSR-Sb.}, vol.~13, pp.~301--312, 1971.

\bibitem{Plykin1984}
R.~V. Plykin, ``On the geometry of hyperbolic attractors of smooth cascades,''
  {\em Russian Mathematical Surveys}, vol.~39, no.~6, p.~85, 1984.

\bibitem{MeZhu2002}
E.~V. Zhuzhoma and V.~S. Medvedev, ``On non-orientable two-dimensional basic
  sets on 3-manifolds,'' {\em Sbornik: Mathematics}, vol.~193, no.~6,
  pp.~869--888, 2002.

\bibitem{KaMaYo1984}
J.~Kaplan, J.~Mallet-Paret, and J.~Yorke, ``The lyapunov dimension of nowhere
  differentiable attracting torus,'' {\em Erhodic Theory Dynamical Systems},
  vol.~4, no.~2, pp.~261--281, 1984.

\bibitem{Bothe}
H.~Bothe, ``The ambient structure of expanding attractors, local triviality,
  tubular neighborhoods,'' {\em Mathematische Nachrichten}, vol.~107, no.~1,
  pp.~327--348, 1982.

\bibitem{GrZh2002}
V.~Z. Grines and E.~V. Zhuzhoma, ``Structurally stable diffeomorphisms with
  basis sets of codimension one,'' {\em Izvestiya: Mathematics}, vol.~66,
  no.~2, p.~223, 2002.

\bibitem{BGPZ}
M.~Barinova, V.~Grines, and O.~Pochinka, ``Dynamics of three-dimensional
  a-diffeomorphisms with two-dimensional attractors and repellers,'' {\em
  Journal of Difference Equations and Applications}, pp.~1--12, 2022.

\bibitem{GrLeMePo2015}
V.~Grines, Y.~Levchenko, V.~Medvedev, and O.~Pochinka, ``The topological
  classification of structurally stable 3-diffeomorphisms with two-dimensional
  basic sets,'' vol.~28, pp.~4081--4102, oct 2015.

\bibitem{Jia}
B.~Jiang, Y.~Ni, and S.~Wang, ``3-manifolds that admit knotted solenoids as
  attractors,'' {\em Transactions of the American Mathematical Society},
  vol.~356, no.~11, pp.~4371--4382, 2004.

\bibitem{GrMePo2016}
V.~Grines, T.~Medvedev, and O.~Pochinka, {\em Dynamical Systems on 2- and
  3-Manifolds}, vol.~46.
\newblock 01 2016.

\bibitem{Do}
A.~Dold, ``Fixed point index and fixed point theorem for euclidean neighborhood
  retracts,'' {\em Topology}, vol.~4, no.~1, pp.~1--8, 1965.

\bibitem{Franksdim}
J.~M. Franks, ``The dimension of basic sets,'' {\em Journal of Differential
  Geometry}, vol.~12, no.~3, pp.~435--441, 1977.

\bibitem{Dold}
A.~Dold, {\em Lectures on algebraic topology}.
\newblock Springer Science \& Business Media, 2012.

\bibitem{JahrenKwasik2011}
B.~Jahren and S.~Kwasik, ``Free involutions on $s^1\times s^n$,'' {\em
  Mathematische Annalen}, vol.~351, pp.~281--303, 2011.

\end{thebibliography}
\end{document}